\documentclass[11pt,oneside,reqno]{amsart}

\usepackage{tikz}
\usetikzlibrary{arrows.meta}
\usetikzlibrary{bending}

\usepackage{hyperref}
\usepackage{amsxtra}
\usepackage{amsopn}
\usepackage{amsmath,amsthm,amssymb}
\usepackage{color}
\usepackage{amscd}
\usepackage{pifont}
\usepackage{amsfonts}
\usepackage{latexsym}
\usepackage{verbatim}
\usepackage{pb-diagram}
\usepackage{tikz}

\newcommand{\iu}{{i\mkern1mu}}

\newcommand{\fr}{\mathfrak}

\newcommand{\op}{\operatorname}

 \newtheorem{lemma} {Lemma} [section]
\newtheorem{theorem}[lemma]{Theorem} 
 
\newtheorem{prop} [lemma]{Proposition}  
\newtheorem{definition}[lemma] {Definition} 
\newtheorem{corol}[lemma] {Corollary}

\begin{document}
\title{Geodesic orbit spaces of compact Lie groups of rank two} 
\author{Nikolaos Panagiotis Souris}\thanks{The author wishes to thank Prof. Y. Nikonorov for pointing out the omission of the Berger spheres in the original version of the manuscript}. 
\address{University of Patras, Department of Mathematics, University Campus, 26504, Rio Patras, Greece}
\email{nsouris@upatras.gr}

\begin{abstract}
Geodesic orbit spaces are those Riemannian homogeneous spaces $(G/H,g)$ whose geodesics are orbits of one-parameter subgroups of $G$.  We classify the simply connected geodesic orbit spaces where $G$ is a compact Lie group of rank two.  We prove that the only such spaces for which the metric $g$ is not induced from a bi-invariant metric on $G$ are certain spheres and projective spaces, endowed with metrics induced from Hopf fibrations.



\end{abstract}
\maketitle

\section{Introduction}

Geodesic orbit spaces (or g.o. spaces) $(G/H,g)$ are natural generalizations of symmetric spaces, satisfying the property that their geodesics are orbits of one parameter subgroups of $G$.  Equivalently, for any geodesic $\gamma$ through the origin $o=eH$ there exists a non-zero vector $X$ in the Lie algebra $\fr{g}$ of $G$ such that 

\begin{equation*}\gamma(t)=\exp(tX)\cdot o,\end{equation*}

\noindent where $\exp:\fr{g}\rightarrow G$ is the exponential map and $\cdot$ denotes the action of $G$ on $G/H$.  The $G$-invariant metric $g$ is also called geodesic orbit (or g.o. metric). Geodesic orbit spaces were initially considered in \cite{KoVa}; since then they are extensively studied within the Riemannian, pseudo-Riemannian and Finsler framework (see the recent studies \cite{Ar1}, \cite{CaZa}, \cite{CheCheWo}, \cite{GoNi}, \cite{Ni}, \cite{YaDe} and references therein), while their complete classification remains open.  On the other hand, important partial classification results were obtained among other works in \cite{AlAr}, \cite{AlNi}, \cite{CheNi}, \cite{Gor1} and \cite{Ta}.

Apart from symmetric spaces, other subclasses of g.o. spaces include \emph{weakly symmetric spaces} $(G/H,g)$ with $G$ being the full isometry group (\cite{BeKoVa}, \cite{Wo2}), \emph{isotropy irreducible spaces} (\cite{Wo}), \emph{$\delta$-homogeneous spaces} (\cite{BeNi1}) as well as \emph{Clifford-Wolf homogeneous spaces} (\cite{BeNi2}).  Another important subclass whose complete description also remains open is that of \emph{naturally reductive spaces} (see the recent works \cite{Ag}, \cite{Ol}, \cite{St} and references therein).  Most known g.o. metrics are naturally reductive with the prime example being the \emph{normal metrics}, that is those metrics induced from bi-invariant metrics on $G$. If $G$ is compact then $G/H$ admits at least one normal (and hence g.o.) metric $g$. If $G$ is semisimple and $g$ is induced from the negative of Killing form of $\fr{g}$ then $g$ is called \emph{standard}.

In this paper we classify the simply connected Riemannian g.o. spaces $(G/H,g)$, where $G$ is a compact Lie group of rank two and $g$ is non-normal.  It turns out that the only spaces satisfying the above requirements are spheres and projective spaces, with the corresponding metrics being deformation metrics along fibers of Hopf fibrations.  More specifically, the classification is given by the following theorem.

\begin{theorem}\label{main}
 The following are the only simply connected Riemannian geodesic orbit spaces $(G/H, g)$ such that $G$ is a compact Lie group of rank two and $g$ is not a normal metric: \\
 
\noindent \emph{(i)} The sphere $(S^3=U(2)/U(1),g^{\lambda}$), $1\neq \lambda>0$; here $g^{\lambda}$ denotes the one-parameter family of deformations of the standard metric $g^1$ along the fiber $S^1$ of the Hopf fibration of $S^3$ on the complex projective space $\mathbb CP^1$.  The metrics $g^{\lambda}$, $\lambda>0$, exhaust the $U(2)$-invariant metrics on $S^3$.\\
 
 \emph{(ii)} The spheres $(S^5=SU(3)/SU(2), g^{\lambda}$), $1\neq \lambda>0$; here $g^{\lambda}$ denotes the one-parameter family of deformations of the standard metric $g^1$ along the fiber $S^1$ of the Hopf fibration of $S^5$ on the complex projective space $\mathbb CP^2$.  The metrics $g^{\lambda}$, $\lambda>0$, exhaust the $SU(3)$-invariant metrics on $S^5$.\\
 
 \emph{(iii)}  The complex projective space $\big(\mathbb CP^3=Sp(2)/(U(1)\times Sp(1)),g^{\lambda}\big)=(SO(5)/U(2),g^{\lambda})$, $1\neq \lambda>0$; here $g^{\lambda}$ denotes the one-parameter family of deformations of the standard metric $g^1$ along the fiber $\mathbb C P^1$ of the Hopf fibration of $\mathbb C P^3$ on the quaternionic projective space $\mathbb H P^1$.  The metrics $g^{\lambda}$, $\lambda>0$, exhaust the $Sp(2)$-invariant metrics on $\mathbb CP^3$.\\
 
 \emph{(iv)} The sphere $(S^7=Sp(2)/Sp(1),g^{\lambda})$, $1\neq \lambda>0$; here $g^{\lambda}$ denotes the one-parameter family of deformations of the standard metric $g^1$ along the fiber $S^3$ of the Hopf fibration of $S^7$ on the quaternionic projective space $\mathbb H P^1$. The metrics $g^{\lambda}$, $\lambda>0$, exhaust the $Sp(2)\times Sp(1)$-invariant metrics on $S^7$.
\end{theorem}

We remark that the aforementioned g.o. spaces are known in the literature, and for certain values of the parameter $\lambda$ they constitute examples of $\delta$-homogeneous manifolds (\cite{BeNis}). The first three spaces are also weakly symmetric manifolds.  Finally, we have the following corollary of Theorem \ref{main} for the compact Lie group $G_2$.

\begin{corol}\label{CDM} A homogeneous Riemannian space of the form $(G_2/H,g)$ is geodesic orbit if and only if $g$ is the standard metric.\end{corol}

\noindent Corollary \label{CDM} follows from Theorem \ref{main} and Corollary \ref{UnivCover1}. 

\subsection{Overview of the proof}
To prove Theorem \ref{main}, we firstly reduce the classification to those simply connected g.o. spaces $(G/H,g)$ such that $G$ is semisimple and the Lie algebra $\fr{h}$ of $H$ is isomorphic to $\fr{su}(2)$.  To obtain this reduction, we take into account the classification of the compact, simply connected g.o. spaces of positive Euler characteristic (\cite{AlNi}) as well as the recent characterization of the g.o. spaces $(G/H,g)$ with $G$ compact semisimple and $H$ abelian (\cite{So2}). 

 For the description of the spaces $G/H$ with $G$ semisimple and $\fr{h}=\fr{su}(2)$, we take into account the explicit description of the embeddings of $\fr{sl}_2\mathbb C$ in the complexified Lie algebra $\fr{g}^{\mathbb C}$ of $G$ in \cite{DoRe} (see also Table~I).  The corresponding embedding of $\fr{su}(2)$ in $\fr{g}$ is shown in Table~II.  For most spaces $G/H$ with $\fr{h}=\fr{su}(2)$, we explicitly calculate the \emph{isotypic decomposition} of the \emph{isotropy representation} of $H$ on the tangent space $T_o(G/H)$ (see further details in subsection \ref{IsotSub}).  The isotypic decomposition allows the explicit description of all $G$-invariant metrics on $G/H$.
 
   A useful tool in identifying the candidate g.o. metrics among the $G$-invariant metrics is Lemma \ref{NormalizerLemma}, combined with Lemma \ref{DualNormalizer}; These lemmas state that any g.o. metric on $G/H$ induces a g.o. metric on the homogeneous space $G/N_G(H^0)$ and a bi-invariant metric on the compact group $N_G(H^0)/H^0$ (here $H^0$ denotes the identity component of $H$ and $N_G(H^0)$ denotes the normalizer of $H^0$ in $G$).
  
  To decide whether a given metric is geodesic orbit, we also take into account the recent classification of the g.o. spaces whose isotropy representation decomposes into two irreducible summands in \cite{CheNi},  the classification of the g.o. metrics on spheres in \cite{Nik0} as well as the classification of the g.o. spaces fibered over irreducible symmetric spaces in \cite{Ta}. Finally, a central argument for the case $G=G_2$ is Proposition \ref{mainargument}, where we show that if the Lie algebra of $N_{G_2}(H)/H$ is isomorphic to $\fr{su}(2)$ then any g.o. metric on $G_2/H$ is standard.  

In Section \ref{Prel} we state some preliminary facts for invariant metrics on homogeneous spaces and the isotropy representation (subsections \ref{subsection1} and \ref{IsotSub}).   We also mention and derive some useful simplification results for geodesic orbit metrics (subsection \ref{GOMET}), while we list some special subgroups $H$ of $G$ for which the g.o. metrics on $G/H$ are known in the literature (subsection \ref{CategoriesSpecial}).  Section \ref{PrelG2} contains preliminary results about the root structure of the compact semisimple Lie algebras $\fr{g}$ of rank two, the embeddings of $\fr{sl}_2 \mathbb C$ in $\fr{g}^{\mathbb C}$ and the corresponding embeddings of $\fr{su}(2)$ in $\fr{g}$.  Finally, in Section \ref{proof} we prove Theorem \ref{main}.

\section{Preliminaries on compact homogeneous spaces and geodesic orbit spaces}\label{Prel}

\subsection{Homogeneous spaces $G/H$ with $G$ compact, and $G$-invariant metrics}\label{subsection1}
Let $G/H$ be a homogeneous space with $G$ compact, and let $o=eH$ be its origin.  Since the isotropy subgroup $H$ is closed in $G$, it is also compact. We denote by $\fr{g},\fr{h}$ the Lie algebras of $G,H$ respectively.  Let $\op{Ad}:G\rightarrow \op{Aut}(\fr{g})$ be the adjoint representation of $G$ and let $\op{ad}:\fr{g}\rightarrow \op{End}(\fr{g})$ be the adjoint representation of $\fr{g}$, where $\op{ad}_XY=[X,Y]$.  Since $G$ is compact there exists an $\op{Ad}$-invariant inner product $Q$ on $\fr{g}$, which we will henceforth fix (for compact semisimple Lie algebras, the negative of the Killing form is such a product).  The Lie algebra $\fr{g}$ admits a $Q$-orthogonal decomposition

\begin{equation*}\label{Dec}\fr{g}=\fr{h}\oplus \fr{m},\end{equation*}

\noindent where $\fr{m}$ is $\op{Ad}_H$-invariant (and hence $\op{ad}_{\fr{h}}$-invariant).  The space $\fr{m}$ can be naturally identified with the tangent space $T_o(G/H)$.  A Riemannian metric $g$ on $G/H$ is called \emph{$G$-invariant} if for all $x\in G$ the left translations $\tau_x:G/H\rightarrow G/H$, $yH\mapsto (xy)H$, are isometries of $(G/H,g)$.  The $G$-invariant metrics on $G/H$ are in one to one correspondence with $\op{Ad}_H$-invariant inner products $\langle \ ,\ \rangle$ on $\fr{m}$, and the latter are in one to one correspondence with endomorphisms $\Lambda:\fr{m}\rightarrow \fr{m}$ satisfying

\begin{equation}\label{MetEnd}\langle X,Y \rangle=Q(\Lambda X,Y), \ \ X,Y\in \fr{m}.\end{equation}    

\noindent An endomorphism $\Lambda\in \op{End}(\fr{m})$ satisfying the above equation for some $\op{Ad}_H$-invariant inner product $\langle \ ,\ \rangle$ is called a \emph{metric endomorphism} and defines a unique $G$-invariant metric on $G/H$.  It follows that any metric endomorphism $\Lambda$ is symmetric with respect to $Q$, positive definite and $\op{Ad}_H$-equivariant, that is $\op{Ad}_h\circ \Lambda=\Lambda\circ \op{Ad}_h$ for all $h\in H$.  Since $\Lambda$ is $\op{Ad}_H$-equivariant, it is also $\op{ad}_{\fr{h}}$-equivariant and the converse holds if $H$ is connected.     

\subsection{The isotropy representation and the form of the metric endomorphisms}\label{IsotSub}
The \emph{isotropy representation} $\op{Ad}^{G/H}:H\rightarrow \op{Gl}(\fr{m})$ is the restriction to $\fr{m}$ of the adjoint representation of $H$ on $\fr{g}$, i.e. $\op{Ad}^{G/H}(h)X:=\op{Ad}_hX$ for $h\in H$ and $X\in \fr{m}$.  If $H$ is connected, $\op{Ad}^{G/H}$ is completely determined by the corresponding isotropy algebra representation $\op{ad}^{\fr{g}/\fr{h}}:\fr{h}\rightarrow \op{End}(\fr{m})$, given by $\op{ad}^{\fr{g}/\fr{h}}(a)X:=\op{ad}_aX=[a,X]$ for $a\in \fr{h}$ and $X\in \fr{m}$.

The space $\fr{m}$ admits a $Q$-orthogonal decomposition into irreducible $\op{Ad}^{G/H}$-submodules, and the pairwise equivalent submodules comprise the isotypic components of $\op{Ad}^{G/H}$. More specifically, a subspace $\fr{p}$ of $\fr{m}$ is called an \emph{isotypic component} of $\op{Ad}^{G/H}$ if the following two conditions hold:\\
 (i) $\fr{p}=\fr{m}_1\oplus \cdots \oplus \fr{m}_l$ where $\fr{m}_j$, $j=1,\dots ,l$, are pairwise equivalent, irreducible $\op{Ad}^{G/H}$-submodules.\\
 (ii) If $\fr{n}\subseteq \fr{m}$ is a submodule of $\op{Ad}^{G/H}$ which is equivalent to $\fr{m}_j$, $j=1,\dots,l$, then $\fr{n}\subseteq \fr{p}$.\\
   The tangent space $\fr{m}$ admits a unique $Q$-orthogonal decomposition

\begin{equation*}\label{isotyp}\fr{m}=\fr{p}_1\oplus \cdots \oplus \fr{p}_s,\end{equation*}

\noindent called the \emph{isotypic decomposition of $\op{Ad}^{G/H}$}, where $\fr{p}_1,\dots, \fr{p}_s$ are the isotypic components of $\op{Ad}^{G/H}$.  If $H$ is connected, then the isotypic decomposition of $\op{Ad}^{G/H}$ coincides with the isotypic decomposition of $\op{ad}^{\fr{g}/\fr{h}}$.  Any metric endomorphism $\Lambda\in \op{End}(\fr{m})$ admits the block-diagonal form

\begin{equation*}\Lambda=\begin{pmatrix} 
 \left.\Lambda\right|_{\fr{p}_1} & 0 & \cdots &0\\
 0& \left.\Lambda\right|_{\fr{p}_2} &\cdots &0\\
  \vdots & \cdots & \ddots &\vdots\\
  0&\cdots &\cdots &\left.\Lambda\right|_{\fr{p}_s}
  \end{pmatrix}.
\end{equation*}
   
\noindent In particular, $\Lambda\fr{p}_j\subseteq \fr{p}_j$.  Moreover, if an isotypic component $\fr{p}_j$ is $\op{Ad}_H$-irreducible (and thus $\op{Ad}^{G/H}$-irreducible) then $\left.\Lambda\right|_{\fr{p}_j}=\lambda_j\op{Id}$.  The same conclusion is true if $H$ is connected and $\fr{p}_j$ is $\op{ad}_{\fr{h}}$-irreducible.

An important isotypic component of the representation $\op{ad}^{\fr{g}/\fr{h}}:\fr{h}\rightarrow \op{End}(\fr{m})$ is the Lie algebra 

\begin{equation}\label{palgebra}\fr{p}=\{X\in \fr{m}: [a,X]=0\ \ \makebox{for all} \ \ a\in \fr{h}\},\end{equation}

\noindent consisting of those elements in $\fr{m}$ to which $\fr{h}$ acts trivially.  The irreducible $\op{ad}^{\fr{g}/\fr{h}}$-submodules comprising $\fr{p}$ are one-dimensional.  For a group $K$, denote by $K^0$ its identity component. We have the following.  

\begin{lemma}\label{Katak} The isotypic component $\fr{p}$ defined by relation \eqref{palgebra} coincides with the Lie algebra $\fr{n}_{\fr{g}}(\fr{h})/\fr{h}$ of the compact Lie group $N_G(H^0)/H^0$, where $\fr{n}_{\fr{g}}(\fr{h})=\{X\in \fr{g}: [a,X]\in \fr{h}\ \ \makebox{for all} \ \ a\in \fr{h}\}$. 
\end{lemma}
\begin{proof} Consider the $Q$-orthogonal decomposition $\fr{n}_{\fr{g}}(\fr{h})=\fr{h}\oplus \fr{r}$.  It suffices to show that $\fr{p}=\fr{r}$.  In view of the $Q$-orthogonal decomposition $\fr{g}=\fr{h}\oplus \fr{m}$, we have $\fr{r}=\fr{n}_{\fr{g}}(\fr{h})\cap \fr{m}$. Given that $\fr{p}$ is a subset of both $\fr{n}_{\fr{g}}(\fr{h})$ and $\fr{m}$, we obtain $\fr{p}\subseteq \fr{r}$.  On the other hand, since $\fr{n}_{\fr{g}}(\fr{h})$ normalizes $\fr{h}$ we have $[\fr{r},\fr{h}]\subseteq \fr{h}$, while the $\op{Ad}$-invariance of the product $Q$ yields $Q([\fr{r},\fr{h}],\fr{h})\subseteq Q(\fr{r},[\fr{h},\fr{h}])\subseteq Q(\fr{r},\fr{h})=\{0\}$, and thus $[\fr{r},\fr{h}]\subseteq \fr{r}$.  The last inclusion along with inclusion $[\fr{r},\fr{h}]\subseteq \fr{h}$ and the $Q$-orthogonality of $\fr{h}$ and $\fr{r}$ yield $[\fr{r},\fr{h}]=\{0\}$. Since $\fr{r}\subset \fr{m}$, the last relation yields $\fr{r}\subseteq \fr{p}$.\end{proof}

\subsection{Geodesic orbit metrics}\label{GOMET}

\begin{definition}
A $G$-invariant metric $g$ on $G/H$ is called a geodesic orbit metric (or a g.o. metric) if any geodesic of $(G/H,g)$ is an orbit of a one parameter subgroup of $G$.  Equivalently, $g$ is a geodesic orbit metric if for any geodesic $\gamma$ of $(G/H,g)$ through the origin $o$ there exists a non-zero vector $X\in \fr{g}$ such that $\gamma(t)=\exp (tX)\cdot o$, $t\in \mathbb R$.  The space $(G/H,g)$ is called a geodesic orbit space (or g.o. space). 

\end{definition}

Let $G/H$ be a homogeneous space with $G$ compact.  We henceforth fix an $\op{Ad}$-invariant inner product $Q$ on $\fr{g}$ and we consider the $Q$-orthogonal reductive decomposition $\fr{g}=\fr{h}\oplus \fr{m}$. We identify each $G$-invariant metric on $G/H$ with the corresponding metric endomorphism $\Lambda\in \op{End}(\fr{m})$.  The following is a necessary and sufficient condition for $\Lambda$ to define a g.o. metric.

\begin{prop}\emph{(\cite{AlAr}, \cite{So1})}\label{GOCond} The metric endomorphism $\Lambda\in \op{End}(\fr{m})$ defines a geodesic orbit metric on $G/H$ if and only if for any vector $X\in\fr{m}\setminus \left\{ {0} \right\}$ there exists a vector $a\in \fr{h}$ such that  

\begin{equation}\label{cor}[a+X,\Lambda X]=0.\end{equation}
\end{prop}

The prime example of a g.o. metric is the metric induced from an $\op{Ad}$-invariant inner product on $\fr{g}$.  More specifically, a metric $g$ on $G/H$ is called \emph{normal} if there exists an $\op{Ad}$-invariant inner product $Q_0$ on $\fr{g}$ and a $Q_0$-orthogonal decomposition $\fr{g}=\fr{h}\oplus \fr{m}$ such that the corresponding metric endomorphism $\Lambda\in \op{End}(\fr{m})$ is a scalar multiple of the identity.  If $Q_0$ is the negative of the Killing form then the metric $g$ is called \emph{standard}.

There are several necessary conditions that simplify the form of $\Lambda$, given that the latter defines a g.o. metric.  One of the most important of those results is the following. 

\begin{lemma}\label{NormalizerLemma}\emph{(\cite{Ni})}
The inner product $\langle \ ,\ \rangle$, generating the metric of a g.o. Riemannian space $(G/H,g)$, is not only $\op{Ad}_H$-invariant but also $\op{Ad}_{N_G(H^0)}$-invariant.  In particular, if $(G/H,g)$ is a g.o. space then $G/N_G(H^0)$, endowed with the induced metric from $g$, is also a g.o. space.\end{lemma}

\begin{corol}\label{NormalizerCorol}Let $\Lambda\in \op{End}(\fr{m})$ be the metric endomorphism of a g.o. metric on $G/H$ and let $\fr{q}\subset \fr{m}$ be the tangent space $T_o(G/N_G(H^0))$.  Then the restriction $\left.\Lambda\right|_{\fr{q}}$ defines a g.o metric on $G/N_G(H^0)$. 
\end{corol}

The following can be considered as a complementary result to Lemma \ref{NormalizerLemma}; It shows that if $(G/H,g)$ is a g.o. space with $G$ compact then the compact Lie group $N_G(H^0)/H^0$, endowed with the induced metric, is also a g.o. space. 

\begin{lemma}\label{DualNormalizer}
Let $(G/H,g)$ be a g.o. space with $G$ compact and with corresponding metric endomorphism $\Lambda$.  Then the restriction of $\Lambda$ to the Lie algebra $\fr{p}\subset \fr{m}$ of the compact Lie group $N_G(H^0)/H^0$ defines a bi-invariant metric on $N_G(H^0)/H^0$.
\end{lemma}

\begin{proof} The Lie algebra $\fr{n}_{\fr{g}}(\fr{h})$ of $N_G(H^0)$ admits a $Q$-orthogonal decomposition $\fr{n}_{\fr{g}}(\fr{h})=\fr{h}\oplus \fr{p}$, where $\fr{p}$ is the Lie algebra of $N_G(H^0)/H^0$ (see also Lemma \ref{Katak}).  We also have a $Q$-orthogonal decomposition $\fr{m}=\fr{p}\oplus \fr{q}$, where $\fr{q}$ can be identified with the tangent space of $G/N_G(H^0)$ at the origin.  Corollary \ref{NormalizerCorol} implies that $\left.\Lambda\right|_{\fr{q}}$ defines a g.o. metric on $G/N_G(H^0)$ and hence $\Lambda\fr{q}\subseteq \fr{q}$.  Along with the symmetry of $\Lambda$ with respect to $Q$, we obtain $Q(\Lambda\fr{p},\fr{q})=Q(\fr{p},\Lambda\fr{q})\subseteq Q(\fr{p},\fr{q})=\{0\}$.  The last equation along with decomposition $\fr{m}=\fr{p}\oplus \fr{q}$ yield $\Lambda\fr{p}\subseteq \fr{p}$, and hence the restriction $\left.\Lambda\right|_{\fr{p}}:\fr{p}\rightarrow \fr{p}$ defines a left-invariant metric on $N_G(H^0)/H^0$.  Since $\Lambda$ defines a g.o. metric on $G/H$, Proposition \ref{GOCond} implies that for any $X\in \fr{p}$ there exists a vector $a\in \fr{h}$ such that $0=[a+X,\Lambda X]=[a+X,\left.\Lambda\right|_{\fr{p}}X]$, and thus $\left.\Lambda\right|_{\fr{p}}$ defines a g.o. metric on $N_G(H^0)/H^0$.  But any left-invariant g.o. metric on a Lie group is necessarily bi-invariant (\cite{AlNi}), and thus $\left.\Lambda\right|_{\fr{p}}$ defines a bi-invariant metric on $N_G(H^0)/H^0$.\end{proof}

The importance of Lemma \ref{DualNormalizer} lies in the fact that the g.o. (i.e. the bi-invariant) metrics on compact Lie groups have a simple description.  Recall that if $G$ is a compact Lie group then its Lie algebra $\fr{g}$ has the direct sum decomposition $\fr{g}=\fr{g}_1\oplus \cdots \oplus \fr{g}_k\oplus \fr{z}(\fr{g})$, where $\fr{g}_j$ are the simple ideals of $\fr{g}$ and $\fr{z}(\fr{g})$ is its center.

\begin{lemma}\label{GOLieGroups}\emph{(\cite{DaZi}, \cite{So1})} 
Let $G$ be a compact Lie group with Lie algebra $\fr{g}=\fr{g}_1\oplus \cdots \oplus \fr{g}_k\oplus \fr{z}(\fr{g})$.  A left-invariant metric $g$ on $G$, with corresponding metric endomorphism $\Lambda$, is a g.o. metric (i.e. a bi-invariant metric) if and only if 

\begin{equation*}\Lambda=\begin{pmatrix} 
 \lambda_1\left.\op{Id}\right|_{\fr{g}_1} & 0 & \cdots &0\\
  \vdots & \ddots & \cdots &\vdots\\
  0&\cdots &\lambda_k\left.\op{Id}\right|_{\fr{g}_k} &0\\
  0& \cdots & 0 & \left.\Lambda\right|_{\fr{z}(\fr{g})}
  \end{pmatrix}, \ \lambda_j>0.
\end{equation*}

\end{lemma}

\subsection{Geodesic orbit spaces $(G/H,g)$ for some special subgroups $H$}\label{CategoriesSpecial}

Let $G$ be a compact, connected semisimple Lie group.  In studying the g.o. spaces of the form $G/H$, it is useful to consider several categories of subgroups $H$ of $G$ such that the g.o. metrics on $G/H$ have been completely described. In particular, the g.o. metrics on $G/H$ are explicitly known in the following cases:\\

\noindent \textbf{1.} The isotropy representation $\op{Ad}^{G/H}:H\rightarrow \op{Gl}(\fr{m})$ is irreducible.   \\
\noindent \textbf{2.} The space $G/H$ is simply connected and the isotropy representation $\op{Ad}^{G/H}:H\rightarrow \op{Gl}(\fr{m})$ decomposes into exactly two irreducible submodules. \\
\noindent \textbf{3.} The space $G/H$ is fibered over an irreducible symmetric space $G/K$ and is endowed with the corresponding fibration metric.  More specifically, there exists a Lie subgroup $K$ with $H\subset K \subset G$ such that: (i) $G/K$ is an irreducible symmetric space and (ii) the $G$-invariant metric on $G/H$ is induced (up to homothety) by an inner product of the form 

\begin{equation}\label{FibForm}g^{\lambda}=\left.(-Q)\right|_{\mathcal{M}_B\times \mathcal{M}_B}+\lambda\left.(-Q)\right|_{\mathcal{M}_F\times \mathcal{M}_F},\end{equation}

\noindent where $\mathcal{M}_B=T_o(G/K)$, $\mathcal{M}_F=T_o(K/H)$ and $Q$ is the Killing form of $\fr{g}$.\\ 
\noindent \textbf{4.} The space $G/H$ is simply connected and $\op{rank}(H)=\op{rank}(G)$ (or equivalently, $G/H$ has positive Euler characteristic).\\
\noindent \textbf{5.} The subgroup $H$ is abelian.\\
\noindent \textbf{6.} The dimension of $G/H$ is less than or equal to six.\\

\noindent More specifically, in Case \textbf{1.} any $G$-invariant metric on $G/H$ is normal and thus any $G$-invariant metric on $G/H$ is g.o.  The isotropy irreducible spaces are classified in \cite{Wo}.  For Case \textbf{2.}, there exist several non-normal g.o. spaces.  These are classified in \cite{CheNi}.  For Case \textbf{3.}, the corresponding spaces are classified in \cite{Ta}.  For Case \textbf{4.}, the corresponding spaces are classified in \cite{AlNi}.  For Case \textbf{5.}, it was recently proven by the author that any g.o. metric on $G/H$, where $G$ is a compact connected semisimple Lie group and $H$ is abelian, is necessarily normal (\cite{So2}).  Finally, the g.o. spaces $(G/H,g)$ with $\dim(G/H)\leq 6$ are classified in \cite{KoVa}.

 We recall that the universal cover of a homogeneous space $G/H$ is the homogeneous space $\widetilde{G}/\widetilde{H}$ where $\widetilde{G}$ is the universal covering group of $G$ and $\widetilde{H}$ is the identity component of $\pi^{-1}(H)$, where $\pi:\widetilde{G}\rightarrow G$ is the canonical projection (\cite{No}).  Moreover, the spaces $G/H$ and $\widetilde{G}/\widetilde{H}$ are locally diffeomorphic.  As cases \textbf{2.} and \textbf{4.} concern simply connected spaces, the following result is useful. 

\begin{prop}\label{UnivCover}
Let $G/H$, $\widetilde{G}/\widetilde{H}$ be homogeneous spaces with $\widetilde{H}$ connected and such that the Lie algebras of $G$ and $\widetilde{G}$ coincide and the Lie algebras of $H$ and $\widetilde{H}$ coincide. \\
\emph{(i)} Any $G$-invariant metric on $G/H$ defines a $\widetilde{G}$-invariant metric on $\widetilde{G}/\widetilde{H}$.\\
\emph{(ii)}  Any $G$-invariant g.o. metric on $G/H$ defines a $\widetilde{G}$-invariant g.o. metric on $\widetilde{G}/\widetilde{H}$.\\
\emph{(iii)} If any invariant g.o. metric on $\widetilde{G}/\widetilde{H}$ is normal (resp. standard) then any invariant g.o. metric on $G/H$ is also normal (resp. standard).\\
Moreover, the converse of each of the statements \emph{(i)} - \emph{(iii)} is true if $H$ is connected.
\end{prop} 
 
\begin{proof}  For convenience, we assume that $G$ is compact although the result is true for any $G$.  Let $\fr{g}$ be the Lie algebra of $G$ and $\widetilde{G}$, and let $\fr{h}$ be the Lie algebra of $H$ and $\widetilde{H}$.  For an (possibly variable) $\op{Ad}$-invariant inner product $Q$ on $\fr{g}$, consider the $Q$-orthogonal decomposition $\fr{g}=\fr{h}\oplus \fr{m}_Q$.  Then $\fr{m}_Q$ can be identified with both tangent spaces $T_o(G/H)$ and $T_o(\widetilde{G}/\widetilde{H})$.  For part (i), any $G$-invariant metric $g$ on $G/H$ corresponds to an $\op{Ad}_H$-invariant inner product $\langle \ ,\ \rangle$ on $\fr{m}_{Q}$ and thus any operator $\op{ad}_a$, $a\in \fr{h}$, is skew-symmetric with respect to $\langle \ ,\  \rangle$.  Since $\widetilde{H}$ is connected, the skew-symmetry of $\op{ad}_{a}$ implies that $\langle \ ,\ \rangle$ is $\op{Ad}_{\widetilde{H}}$-invariant and thus defines a $\widetilde{G}$-invariant metric on $\widetilde{G}/\widetilde{H}$.

 For part (ii), assume that $g$ is a $G$-invariant g.o. metric on $G/H$ and let $\Lambda_Q\in \op{End(\fr{m}_Q)}$ be the corresponding metric endomorphism of $g$ satisfying Equation \eqref{MetEnd}.  By part (i), $\Lambda_Q$ defines a $\widetilde{G}$-invariant metric on $\widetilde{G}/\widetilde{H}$.  The fact that $\Lambda_Q$ defines a g.o. metric on $G/H$ along with Proposition \ref{GOCond} imply that $\Lambda_Q$ also defines a g.o. metric on $\widetilde{G}/\widetilde{H}$. For part (iii),  let $g$ be a $G$-invariant g.o. metric on $G/H$ with metric endomorphism $\Lambda_Q\in \op{End(\fr{m}_Q)}$.  Part (ii) implies that $\Lambda_Q$ defines a $\widetilde{G}$-invariant g.o. metric on $\widetilde{G}/\widetilde{H}$.  By assumption, $\Lambda_Q$ is normal. Therefore, $\Lambda_{Q_0}=\lambda\op{Id}$ for some $\op{Ad}$-invariant inner product $Q_0$ on $\fr{g}$, which in turn implies that $g$ is normal.   \end{proof}

\begin{corol}\label{UnivCover1}If any $\widetilde{G}$-invariant g.o. metric on the universal cover $\widetilde{G}/\widetilde{H}$ of $G/H$ is normal (resp. standard) then any $G$-invariant g.o. metric on $G/H$ is also normal (resp. standard).
\end{corol}

\section{Compact, semsimple Lie algebras of rank two}\label{PrelG2}

In this section we describe the root structure of compact semisimple Lie algebras $\fr{g}$ of rank two, the embeddings of $\fr{sl}_2\mathbb C$ in the complexified Lie algebra $\fr{g}^{\mathbb C}$ and eventually the embeddings of $\fr{h}=\fr{su}(2)$ in $\fr{g}$ along with the orthogonal complement $\fr{m}$ of $\fr{h}$ in $\fr{g}$.

\subsection{General root structure}\label{rootsection0}

For the results in this subsection we refer to \cite{Hel}.  Let $\fr{g}$ be a compact semisimple Lie algebra of rank two and let $\fr{g}^{\mathbb C}$ be its complexification.   Then $\fr{g}$ is isomorphic to one of the Lie algebras $\fr{su}(3)$, $\fr{su}(2)\oplus \fr{su}(2)=\fr{so}(4)$, $\fr{sp}(2)$, $\fr{g}_2$.  Accordingly, $\fr{g}^{\mathbb C}$ is isomorphic to one of the Lie algebras $A_2=\fr{sl}_3\mathbb C$, $A_1\times A_1=\fr{sl}_2\mathbb C\oplus \fr{sl}_2\mathbb C$, $C_2=\fr{sp}_4\mathbb C$, $\fr{g}_2^{\mathbb C}$. 

For each of the aforementioned algebras $\fr{g}$, let $R\subset \fr{t}^*$ be the root system of $\fr{g}^{\mathbb C}$ with respect to a Cartan subalgebra $\fr{t}$ of $\fr{g}^{\mathbb C}$, let $R^+$ be the set of positive roots and let $\Pi=\{\alpha,\beta\}$ be the set of simple roots.  We have the root decomposition $\fr{g}^{\mathbb C}=\fr{t}\oplus \sum_{\gamma\in R}\fr{g}^{\gamma}$, where

\begin{equation*}\fr{g}^{\gamma}=\{X\in \fr{g}^{\mathbb C}:[a,X]=\gamma(a)X\ \ \makebox{for all} \ \ a\in \fr{t}\}.\end{equation*}

\noindent Let $Q^{\mathbb C}$ and $Q$ denote the Killing forms of $\fr{g}^{\mathbb C}$ and $\fr{g}$ respectively.  The restriction of $Q^{\mathbb C}$ on $\fr{t}$ is non-degenerate and induces a dual form $( \ ,\ )$ on $\fr{t}^{*}$.  For $\gamma,\delta\in \fr{t}^*$, $\delta\neq 0$, set $\langle \gamma, \delta\rangle:=2\frac{(\gamma,\delta)}{(\delta,\delta)}$.  The Cartan matrix of $R$ is the matrix $\begin{pmatrix} \langle \alpha,\alpha\rangle & \langle \alpha,\beta\rangle\\ \langle \beta,\alpha \rangle &\langle \beta, \beta\rangle\end{pmatrix}$.  For any $\gamma\in \fr{t}^*$, let $t_{\gamma}$ be the corresponding covector, defined by the relation $Q^{\mathbb C}(t_{\gamma},a)=\gamma(a)$ for all $a\in \fr{t}$, and set $H_{\gamma}:=\frac{2t_{\gamma}}{(\gamma,\gamma)}$, $\gamma\neq 0$.  
    We consider root elements $E_{\gamma}\in \fr{g}^{\gamma}$ such that the set $\{H_{\alpha},H_{\beta},E_{\gamma}: \gamma\in R\}$ is a Chevalley basis of $\fr{g}^{\mathbb C}$.  The vectors $H_{\gamma},E_{\delta}$, $\gamma,\delta\in R$, satisfy the relations

\begin{equation}\label{bracket}[H_{\gamma},H_{\delta}]=0, \ \   [H_{\gamma},E_{\delta}]=\langle \delta, \gamma \rangle E_{\delta} \ \ \makebox{and} \ \ \ [E_{\gamma},E_{\delta}]=\left\{ \begin{array}{lll}  N_{\gamma,\delta}E_{\gamma+\delta},\ \ \mbox{if}\ \ \gamma+\delta \in R\\
H_{\gamma}, \ \ \makebox{if} \ \ \gamma+\delta=0\\
 0 , \quad \mbox{otherwise}
\end{array}
\right.,
\end{equation}

\noindent where $N_{\gamma,\delta}$ are integers satisfying $N_{\gamma,\delta}=-N_{\delta,\gamma}=-N_{-\gamma,-\delta}$, and $N_{\gamma,\delta}$ is non-zero if and only if $\gamma+\delta\in R$.  For $\gamma\in R^+$, we set 

\begin{equation*}F_{\gamma}:=E_{\gamma}-E_{-\gamma},\ \ G_{\gamma}:=\iu(E_{\gamma}+E_{-\gamma})\ \  \makebox{and}  \ \ \fr{m}_{\gamma}:=\op{span}_{\mathbb R}\{F_{\gamma},G_{\gamma}\}.\end{equation*}

\noindent The set 

\begin{equation*}\mathcal{B}:=\{\iu H_{\alpha},\iu H_{\beta},F_{\gamma},G_{\gamma}: \ \gamma \in R\},\end{equation*}

\noindent is basis of the compact real form $\fr{g}$ of $\fr{g}^{\mathbb C}$.  In particular, $\fr{g}$ admits the $Q$-orthogonal decomposition

\begin{equation}\label{decg_2}\fr{g}=\iu \fr{t} \oplus \bigoplus_{\gamma\in R^+}\fr{m}_{\gamma}, \ \ \makebox{where} \ \ \iu \fr{t}=\op{span}_{\mathbb R}\{\iu H_{\alpha},\iu H_{\beta}\}.
\end{equation}

\noindent The Lie bracket relations between the vectors $F_{\gamma}$ and $G_{\gamma}$ can be calculated from relations \eqref{bracket} and the values of the integers $N_{\gamma,\delta}$.  The relations are shown in the following lemma.

\begin{lemma}\label{BracketLem}
Let $\gamma,\delta \in R$ with $\gamma+\delta\neq 0$ and such that if $\gamma-\delta\in R$ then $\gamma-\delta >0$.  Then

\begin{align*}
[F_{\gamma},F_{\delta}]&= N_{\gamma,\delta}F_{\gamma+\delta}-N_{\gamma,-\delta}F_{\gamma-\delta},  & [F_{\gamma},G_{\delta}]&=N_{\gamma,\delta}G_{\gamma+\delta}+N_{\gamma,-\delta}G_{\gamma-\delta},\\[0pt]
[G_{\gamma},F_{\delta}]&=N_{\gamma,\delta}G_{\gamma+\delta}-N_{\gamma,-\delta}G_{\gamma-\delta}, & [G_{\gamma},G_{\delta}]&=-N_{\gamma,\delta}F_{\gamma+\delta}-N_{\gamma,-\delta}F_{\gamma-\delta},\\[0pt]
[F_{\gamma},G_{\gamma}]&=2\iu H_{\gamma}.
\end{align*}

\end{lemma}

\noindent Consequently, the spaces $\fr{m}_{\gamma}$, $\gamma\in R^+$, satisfy the relation

\begin{equation}\label{ma}
[\fr{m}_{\gamma},\fr{m}_{\delta}]\subseteq \left\{ \begin{array}{ll}  \fr{m}_{\gamma+\delta}\oplus \fr{m}_{|\gamma-\delta|},\ \ \mbox{if}\ \ \gamma+\delta \in R \ \ \makebox{or} \ \ \gamma-\delta \in R\\
\left\{ {0} \right\}, \quad \mbox{otherwise}
\end{array}
\right..
\end{equation}

 Using relations \eqref{bracket} we obtain $Q^{\mathbb C}(E_{\gamma},E_{\delta})=\frac{1}{\langle \gamma,\gamma\rangle}Q^{\mathbb C}([H_{\gamma},E_{\gamma}],E_{\delta})=\frac{1}{2}Q^{\mathbb C}(H_{\gamma},[E_{\gamma},E_{\delta}])$ which, along with the fact that $Q^{\mathbb C}(H_{\gamma},H_{\delta})=\frac{4}{(\gamma,\gamma)(\delta,\delta)}Q^{\mathbb C}(t_{\gamma},t_{\delta})=\frac{4(\gamma,\delta)}{(\gamma,\gamma)(\delta,\delta)}$, implies that

\begin{equation*}\label{Kilform}Q^{\mathbb C}(E_{\gamma},E_{\delta})=\left\{ \begin{array}{ll}  \frac{2}{(\gamma,\gamma)}\ \ \mbox{if}\ \ \gamma+\delta=0\\
0, \quad \mbox{otherwise}
\end{array}
\right..
\end{equation*} 

\noindent As a result and given that $\left.Q^{\mathbb C}\right|_{\fr{g}\times \fr{g}}=Q$, we have $Q(F_{\gamma},G_{\delta})=0$ and

\begin{equation}\label{Kilform1}Q(F_{\gamma},F_{\delta})=\left\{ \begin{array}{ll}  -\frac{4}{(\gamma,\gamma)},\ \ \mbox{if}\ \ \gamma=\delta\\
0 , \quad \mbox{otherwise}
\end{array}
\right., \ \  Q(G_{\gamma},G_{\delta})=\left\{ \begin{array}{ll}  -\frac{4}{(\gamma,\gamma)},\ \ \mbox{if} \ \ \gamma=\delta\\
 0, \quad \mbox{otherwise}
\end{array}
\right.. 
\end{equation}

\subsection{The root system of complex semisimple Lie algebras $\fr{g}^{\mathbb C}$ of rank two}\label{rootsection}  \leavevmode \\

 If $\fr{g}^{\mathbb C}=A_2$, we have $R=\{\pm \alpha,\pm \beta, \pm(\alpha+\beta)\}$, with $R^+=\{\alpha,\beta,\alpha+\beta\}$.  The Cartan matrix of $R$ is $\begin{pmatrix}  2 & -1\\ -1 &2 \end{pmatrix}$, from which we may assume without any loss of generality that\\

 $(\alpha,\alpha)=1=\frac{1}{4}Q^{\mathbb C}(H_{\alpha},H_{\alpha})$, \ $(\alpha,\beta)=-\frac{1}{2}=\frac{1}{4}Q^{\mathbb C}(H_{\alpha},H_{\beta})$ \ and $(\beta,\beta)=1=\frac{1}{4}Q^{\mathbb C}(H_{\beta},H_{\beta})$. \\
 
 \noindent Moreover, $N_{\alpha,\beta}=N_{\beta,-(\alpha+\beta)}=N_{-(\alpha+\beta),\alpha}=1$.\\

If $\fr{g}^{\mathbb C}=A_1\times A_1$, we have $R=\{\pm \alpha,\pm \beta\}$, with $R^+=\{\alpha,\beta\}$.  The Cartan matrix of $R$ is $\begin{pmatrix}  2 & 0\\ 0 &2 \end{pmatrix}$, from which we may assume that \\

$(\alpha,\alpha)=1=\frac{1}{4}Q^{\mathbb C}(H_{\alpha},H_{\alpha})$, \  $(\alpha,\beta)=0=Q^{\mathbb C}(H_{\alpha},H_{\beta})$ \ and $(\beta,\beta)=1=\frac{1}{4}Q^{\mathbb C}(H_{\beta},H_{\beta})$. \\ \\

 If $\fr{g}^{\mathbb C}=C_2$, we have $R=\{\pm \alpha,\pm \beta, \pm(\alpha+\beta),\pm(\alpha+2\beta)\}$, with $R^+=\{\alpha,\beta,\alpha+\beta,\alpha+2\beta\}$.  The Cartan matrix of $R$ is $\begin{pmatrix}  2 & -2\\ -1 &2 \end{pmatrix}$, from which we may assume that\\
 
  $(\alpha,\alpha)=1=\frac{1}{4}Q^{\mathbb C}(H_{\alpha},H_{\alpha})$, \ $(\alpha,\beta)=-\frac{1}{2}=\frac{1}{8}Q^{\mathbb C}(H_{\alpha},H_{\beta})$ \ and $(\beta,\beta)=\frac{1}{2}=\frac{1}{16}Q^{\mathbb C}(H_{\beta},H_{\beta})$. \\
  
  \noindent Moreover, $N_{\alpha,\beta}=N_{-(\alpha+\beta),\alpha}=N_{-(\alpha+2\beta),\alpha+\beta}=N_{\beta,-(\alpha+2\beta)}=1$ and $N_{\beta,-(\alpha+\beta)}=N_{\alpha+\beta, \beta}=2$.\\ \\

  Finally, if $\fr{g}^{\mathbb C}=\fr{g}_2$ we have $R=\{\pm \alpha, \pm \beta, \pm (\alpha+\beta),\pm (2\alpha+\beta), \pm (3\alpha+\beta),\pm (3\alpha+2\beta)\}$ with \\
 $R^+=\{\alpha, \beta, \alpha+\beta, 2\alpha+\beta, 3\alpha+\beta, 3\alpha+2\beta\}$.  The Cartan matrix of $R$ is $\begin{pmatrix}  2 & -1\\ -3 &2 \end{pmatrix}$, from which we may assume that\\
 
  $(\alpha,\alpha)=1=\frac{1}{4}Q^{\mathbb C}(H_{\alpha},H_{\alpha})$, \ $(\alpha,\beta)=-\frac{3}{2}=\frac{3}{4}Q^{\mathbb C}(H_{\alpha},H_{\beta})$\  and $(\beta,\beta)=3=\frac{9}{4}Q^{\mathbb C}(H_{\beta},H_{\beta})$. Moreover, we have (see also \cite{Ma})

\begin{eqnarray*}
N_{\beta,\alpha}&=&N_{\beta,3\alpha+\beta}=N_{3\alpha+\beta,-(3\alpha+2\beta)}=N_{2\alpha+\beta,-(3\alpha+\beta)}=N_{2\alpha+\beta,-(3\alpha+2\beta)}\\
&=&N_{-(3\alpha+2\beta),\alpha+\beta}=N_{-(3\alpha+2\beta),\beta}=N_{-(3\alpha+\beta),\alpha}=N_{-(\alpha+\beta),\beta}=1,\\ 
N_{\alpha+\beta,\alpha}&=&N_{\alpha,-(2\alpha+\beta)}=N_{-(2\alpha+\beta),\alpha+\beta}=2, \ \ N_{\alpha,2\alpha+\beta}=N_{\alpha,-(\alpha+\beta)}=N_{\alpha+\beta,2\alpha+\beta}=3.
\end{eqnarray*}

\subsection{Embeddings of $\fr{su}(2)$ in $\fr{g}$.}  Assume that $\fr{h}=\fr{su}(2)$. The following table shows the embeddings of $\fr{h}^{\mathbb C}=A_1=\fr{sl}_2\mathbb C$ in $\fr{g}^{\mathbb C}$ (up to conjugation by inner automorphisms), where $\fr{g}$ is a semisimple Lie algebra of rank two.  The main sources for this table are \cite{DoRe}, \cite{DoRe0} and \cite{Ma}.

\begin{center}
{\renewcommand{\arraystretch}{0.8}
\begin{tabular}{|c|c|}
\hline $\fr{h}^{\mathbb C}=\fr{sl}_2\mathbb C$ &  $\fr{g}^{\mathbb C}$ $\vphantom{\displaystyle{A^{B^{C}}}}$\\
\hline  $\op{span}_{\mathbb C}\{E_{\alpha+\beta},E_{-(\alpha+\beta)},H_{\alpha}+H_{\beta}\}$ &  $A_2$  $\vphantom{\displaystyle{A^{B^{C^{D}}}}}$\\
\hline  $\op{span}_{\mathbb C}\{E_{\alpha}+E_{\beta},E_{-\alpha}+E_{-\beta}, H_{\alpha}+H_{\beta}\}$ &  $A_2$  $\vphantom{\displaystyle{A^{B^{C^{D}}}}}$\\
\hline  $\op{span}_{\mathbb C}\{E_{\alpha},E_{-\alpha},H_{\alpha}\}$ &  $A_1\times A_1$  $\vphantom{\displaystyle{A^{B^{C^{D}}}}}$\\
\hline  $\op{span}_{\mathbb C}\{E_{\beta},E_{-\beta},H_{\beta}\}$ &  $A_1\times A_1$  $\vphantom{\displaystyle{A^{B^{C^{D}}}}}$\\
\hline  $\op{span}_{\mathbb C}\{E_{\alpha}+E_{\beta},E_{-\alpha}+E_{-\beta},H_{\alpha}+H_{\beta}\}$ &  $A_1\times A_1$  $\vphantom{\displaystyle{A^{B^{C^{D}}}}}$\\
\hline  $\op{span}_{\mathbb C}\{E_{\alpha+2\beta},E_{-(\alpha+2\beta)},H_{\alpha}+H_{\beta}\}$ &  $C_2$  $\vphantom{\displaystyle{A^{B^{C^{D}}}}}$\\
\hline  $\op{span}_{\mathbb C}\{E_{\alpha+\beta},E_{-(\alpha+\beta)},2H_{\alpha}+H_{\beta}\}$ &  $C_2$  $\vphantom{\displaystyle{A^{B^{C^{D}}}}}$\\
\hline  $\op{span}_{\mathbb C}\{E_{\alpha}+E_{\beta},4E_{-\alpha}+3E_{-\beta},4H_{\alpha}+3H_{\beta}\}$ &  $C_2$  $\vphantom{\displaystyle{A^{B^{C^{D}}}}}$\\
\hline  $\op{span}_{\mathbb C}\{E_{\alpha},E_{-\alpha},H_{\alpha}\}$ &  $\fr{g}_2^{\mathbb C}$  $\vphantom{\displaystyle{A^{B^{C^{D}}}}}$\\
\hline  $\op{span}_{\mathbb C}\{E_{\beta},E_{-\beta},H_{\beta}\}$ &  $\fr{g}_2^{\mathbb C}$  $\vphantom{\displaystyle{A^{B^{C^{D}}}}}$\\
\hline  $\op{span}_{\mathbb C}\{\sqrt{2}(E_{3\alpha+2\beta}+E_{-\beta}),\sqrt{2}(E_{\beta}+E_{-(3\alpha+2\beta)}),2H_{3\alpha+\beta}\}$ &  $\fr{g}_2^{\mathbb C}$  $\vphantom{\displaystyle{A^{B^{C^{D}}}}}$\\
\hline  $\op{span}_{\mathbb C}\{\sqrt{6}E_{\alpha}+\sqrt{10}E_{\beta},\sqrt{6}E_{-\alpha}+\sqrt{10}E_{-\beta},14H_{9\alpha+5\beta}\}$ &  $\fr{g}_2^{\mathbb C}$  $\vphantom{\displaystyle{A^{B^{C^{D}}}}}$\\
\hline 
\end{tabular} \nopagebreak \\ \nopagebreak {\em Table~I: Embeddings of $\fr{sl}_2\mathbb C$ in complex semisimple Lie algebras $\fr{g}^{\mathbb C}$ of rank two, up to conjugation by inner automorphisms.} $\vphantom{\displaystyle\frac{a}{2}}$

}
\end{center}

The following table shows the corresponding embedding of the compact Lie algebra $\fr{h}$ (up to conjugation by inner automorphisms) in the compact real form $\fr{g}$ of $\fr{g}^{\mathbb C}$.  Moreover, using the values $(\alpha,\alpha)$, $(\alpha,\beta)$ and $(\beta,\beta)$ given in subsection \ref{rootsection},  the definition of the vectors $F_{\gamma},G_{\gamma}$, relations \eqref{Kilform1}, and by taking into account the decomposition \eqref{decg_2} of $\fr{g}$, we explicitly describe the $Q$-orthogonal complement $\fr{m}$ of $\fr{h}$ in $\fr{g}$ in terms of the generating set $\{\iu H_{\gamma},F_{\gamma},G_{\gamma}: \gamma\in R\}$ of $\fr{g}$. 

\begin{center}
{\renewcommand{\arraystretch}{1.5}
\begin{tabular}{|c|c|c|}
\hline $\fr{h}=\fr{su}(2)$    & $\fr{m}=\fr{h}^{\bot}$  &  $\fr{g}$  $\vphantom{\displaystyle{A^{B^{C}}}}$\\
\hline  $\op{span}_{\mathbb R}\{\iu(H_{\alpha}+H_{\beta})\}\oplus \fr{m}_{\alpha+\beta}$   & $\op{span}_{\mathbb R}\{\iu(H_{\alpha}-H_{\beta})\}\oplus \fr{m}_{\alpha}\oplus \fr{m}_{\beta}$ & $\fr{su}(3)$  $\vphantom{\displaystyle{A^{B^{C^{D}}}}}$\\
\hline  $\op{span}_{\mathbb R}\{F_{\alpha}+F_{\beta},G_{\alpha}+G_{\beta},\iu (H_{\alpha}+H_{\beta})\}$ & $\op{span}_{\mathbb R}\{F_{\alpha}-F_{\beta},G_{\alpha}-G_{\beta},\iu(H_{\alpha}-H_{\beta})\}\oplus \fr{m}_{\alpha+\beta}$ & $\fr{su}(3)$ $\vphantom{\displaystyle{A^{B^{C^{D}}}}}$\\
\hline  $\op{span}_{\mathbb R}\{\iu H_{\alpha}\}\oplus \fr{m}_{\alpha}$ & $\op{span}_{\mathbb R}\{\iu H_{\beta}\}\oplus \fr{m}_{\beta}$ & $\fr{so}(4)$ $\vphantom{\displaystyle{A^{B^{C^{D}}}}}$\\
\hline  $\op{span}_{\mathbb R}\{\iu H_{\beta}\}\oplus \fr{m}_{\beta}$ & $\op{span}_{\mathbb R}\{\iu H_{\alpha}\}\oplus \fr{m}_{\alpha}$ & $\fr{so}(4)$ $\vphantom{\displaystyle{A^{B^{C^{D}}}}}$\\
\hline  $\op{span}_{\mathbb R}\{F_{\alpha}+F_{\beta},G_{\alpha}+G_{\beta},\iu (H_{\alpha}+H_{\beta})\}$ & $\op{span}_{\mathbb R}\{F_{\alpha}-F_{\beta}, G_{\alpha}-G_{\beta}, \iu (H_{\alpha}-H_{\beta})\}$ & $\fr{so}(4)$ $\vphantom{\displaystyle{A^{B^{C^{D}}}}}$\\
\hline  $\op{span}_{\mathbb R}\{\iu(H_{\alpha}+H_{\beta})\}\oplus \fr{m}_{\alpha+2\beta}$   & $\op{span}_{\mathbb R}\{\iu H_{\alpha}\}\oplus \fr{m}_{\alpha}\oplus \fr{m}_{\beta}\oplus \fr{m}_{\alpha+\beta}$ & $\fr{sp}(2)$  $\vphantom{\displaystyle{A^{B^{C^{D}}}}}$\\
\hline  $\op{span}_{\mathbb R}\{\iu(2H_{\alpha}+H_{\beta})\}\oplus \fr{m}_{\alpha+\beta}$   & $\op{span}_{\mathbb R}\{\iu H_{\beta}\}\oplus \fr{m}_{\alpha}\oplus \fr{m}_{\beta}\oplus \fr{m}_{\alpha+2\beta}$ & $\fr{sp}(2)$  $\vphantom{\displaystyle{A^{B^{C^{D}}}}}$\\
\hline  $\op{span}_{\mathbb R}\{2F_{\alpha}+\sqrt{3}F_{\beta},2G_{\alpha}+\sqrt{3}G_{\beta},$ & $\op{span}_{\mathbb R}\{\sqrt{3}F_{\alpha}-F_{\beta}, \sqrt{3}G_{\alpha}-G_{\beta}, \iu (2H_{\alpha}-H_{\beta})\}$ & $\fr{sp}(2)$ $\vphantom{\displaystyle{A^{B^{C^{D}}}}}$\\
 $\iu (4H_{\alpha}+3H_{\beta})\}$ &  $\oplus \fr{m}_{\alpha+\beta}\oplus \fr{m}_{\alpha+2\beta}$ & \\
\hline  $\op{span}_{\mathbb R}\{\iu H_{\alpha}\}\oplus \fr{m}_{\alpha}$ & $\op{span}_{\mathbb R}\{\iu H_{3\alpha+2\beta}\}\oplus \fr{m}_{\beta}\oplus \fr{m}_{\alpha+\beta}$ $\vphantom{\displaystyle{A^{B^{C^{D}}}}}$ & $\fr{g}_2$\\
  $\vphantom{\displaystyle{A^{B^{C^{D}}}}}$ & $\oplus \fr{m}_{2\alpha+\beta}\oplus \fr{m}_{3\alpha+\beta}\oplus \fr{m}_{3\alpha+2\beta}$ & $\vphantom{\displaystyle{A^{B^{C^{D}}}}}$ \\
  \hline  $\op{span}_{\mathbb R}\{\iu H_{\beta}\}\oplus \fr{m}_{\beta}$ & $\op{span}_{\mathbb R}\{\iu H_{2\alpha+\beta}\}\oplus \fr{m}_{\alpha}\oplus \fr{m}_{\alpha+\beta}$ $\vphantom{\displaystyle{A^{B^{C^{D}}}}}$ & $\fr{g}_2$\\
  $\vphantom{\displaystyle{A^{B^{C^{D}}}}}$ & $\oplus \fr{m}_{2\alpha+\beta}\oplus \fr{m}_{3\alpha+\beta}\oplus \fr{m}_{3\alpha+2\beta}$ & $\vphantom{\displaystyle{A^{B^{C^{D}}}}}$ \\
  \hline  $\op{span}_{\mathbb R}\{\sqrt{2}(F_{3\alpha+2\beta}-F_{\beta}),\sqrt{2}(G_{3\alpha+2\beta}+G_{\beta}),$ & $\op{span}_{\mathbb R}\{\sqrt{2}(F_{3\alpha+2\beta}+F_{\beta}),\sqrt{2}(G_{3\alpha+2\beta}-G_{\beta}),$ & $\fr{g}_2$ \\
  $2\iu H_{3\alpha+\beta}\}$ & $2\iu H_{\alpha+\beta}\}\oplus \fr{m}_{\alpha}\oplus \fr{m}_{\alpha+\beta}\oplus \fr{m}_{2\alpha+\beta}\oplus \fr{m}_{3\alpha+\beta}$& \\
\hline  $\op{span}_{\mathbb R}\{\sqrt{6}F_{\alpha}+\sqrt{10}F_{\beta},\sqrt{6}G_{\alpha}+\sqrt{10}G_{\beta},$ & $\op{span}_{\mathbb R}\{\sqrt{10}F_{\alpha}-3\sqrt{6}F_{\beta},\sqrt{10}G_{\alpha}-3\sqrt{6}G_{\beta},$ & $\fr{g}_2$ \\
  $14\iu H_{9\alpha+5\beta}\}$ & $2\iu H_{\alpha-\beta}\}\oplus \fr{m}_{\alpha+\beta}\oplus \fr{m}_{2\alpha+\beta}\oplus \fr{m}_{3\alpha+\beta}\oplus \fr{m}_{3\alpha+2\beta}$& \\  
  \hline
\end{tabular} \nopagebreak \\ \nopagebreak {\em Table~II: Embeddings of $\fr{h}=\fr{su}(2)$ in compact semisimple Lie algebras $\fr{g}$ of rank two, up to conjugation by inner automorphisms.} $\vphantom{\displaystyle\frac{a}{2}}$

}
\end{center}

\section{Proof of Theorem \ref{main}}\label{proof}

Let $G$ be a compact Lie group of rank two and let $\fr{g}$ be the Lie algebra of $G$.  For a subgroup $H$ of $G$, let $\fr{h}$ denote its Lie algebra. We will firstly reduce Theorem \ref{main} to the study of those spaces $(G/H,g)$ with $G$ semisimple, simply connected and $\fr{h}=\fr{su}(2)$.

\subsection{Reduction to the case $\fr{h}=\fr{su}(2)$.}

   Assume initially that $G$ is not semisimple.  Since $G$ is reductive and has rank two,  either $G$ is abelian or $\fr{g}=\fr{su}(2)\oplus \op{span}_{\mathbb R}\{X\}=\fr{u}(2)$, where $\op{span}_{\mathbb R}\{X\}$ is the one-dimensional center of $\fr{g}$.  In the first case, any left-invariant metric on $G$ is bi-invariant. Therefore, any $G$-invariant Riemannian metric on a homogeneous space $G/H$ is a normal metric and hence a g.o. metric.  In the second case, and given that $\fr{su}(2)$ does not contain any 2-dimensional Lie subalgebras, the proper Lie subalgebras of $\mathfrak{g}$ are precisely (up to conjugation) the subalgebras $\{0\},\op{span}_{\mathbb R}\{X\}$, $\fr{su}(2)$, $\op{span}_{\mathbb R}\{Y\}\oplus \op{span}_{\mathbb R}\{X\}$, $\op{span}_{\mathbb R}\{Y\}$ or $\op{span}_{\mathbb R}\{X+Y\}$, for some $Y\in \fr{su}(2)$.  The corresponding spaces $G/H$ are locally diffeomorphic to the simply connected spaces $SU(2)\times \mathbb R$, $SU(2)$, $\mathbb R$, $SU(2)/SO(2)$, $(SU(2)/SO(2))\times \mathbb R$ or $U(2)/U(1)$ respectively.  The first three spaces are Lie groups and hence any g.o. metric is bi-invariant, i.e. normal.  The next two spaces are symmetric and thus any invariant metric is normal.  By part (iii) of Proposition \ref{UnivCover}, any g.o metric on $G/H$ is also normal for the first five spaces. Finally, the last space, the Berger sphere $U(2)/U(1)$, is geodesic orbit with respect to any $U(2)$-invariant metric (\cite{Nik0}). The latter metrics have the form 
   
   \begin{equation*}g^{\lambda}=\lambda\left.(-Q)\right|_{\mathcal{M}_F\times \mathcal{M}_F}+\left.(-Q)\right|_{\mathcal{M}_B\times \mathcal{M}_B},\end{equation*}

\noindent where $Q(X,Y)=\op{Tr}(XY)$, $X,Y\in \fr{u}(2)$, $\mathcal{M}_F=T_o\big((U(1)\times U(1)/ U(1)\big)=T_o(S^1)$ and $\mathcal{M}_B=T_o\big(U(2)/(U(1)\times U(1))\big)=T_o(\mathbb C P^1)$.  Hence the metrics $g^{\lambda}$ are deformations of the standard metric $g^1$ along the fiber $\mathbb C P^1$ of the Hopf fibration of $S^3=U(2)/U(1)$ on $\mathbb C P^1$. The space $U(2)/U(1)$ yields part (i) of Theorem \ref{main}. 
   
 Now assume that $G$ is semisimple and let $(G/H,g)$ be a g.o space such that $G/H$ is simply connected.  We recall that the universal cover of $G/H$ is the homogeneous space $\widetilde{G}/\widetilde{H}$, where $\widetilde{G}$ is the universal covering group of $G$ and $\widetilde{H}$ is the identity component of $\pi^{-1}(H)$, where $\pi:\widetilde{G}\rightarrow G$ is the canonical projection.  Moreover, the universal covering group of a compact semisimple Lie group is also compact and semisimple.  Therefore, it suffices to assume that $G$ is simply connected, compact and semisimple, and $H$ is connected. 
 
  Since $H$ is a closed subgroup of $G$, it is compact.  Hence $\fr{h}$ is a compact Lie algebra.  Since the rank of $G$ is two, the dimension of a maximal abelian subalgebra (and thus a Cartan subalgebra) of $\fr{h}$ is at most two.  On the other hand, since $\fr{h}$ is a compact Lie algebra it is also reductive and thus $\fr{h}=\fr{s}\oplus \fr{z}$ where $\fr{s}$ is semisimple and $\fr{z}$ is abelian. If $\fr{h}$ does not have maximal rank then $\fr{h}$ has rank at most one and thus $\fr{s}=\{0\}$ or $\fr{z}=\{0\}$.  In the former case, $\fr{h}$ is abelian of rank (and hence dimension) at most one.  In the latter case, $\fr{h}$ is compact semisimple of rank one, and thus it is isomorphic to $\fr{su}(2)$. Therefore, one of the following cases occurs for $\fr{h}$:\\

\noindent Case I) $\fr{h}$ has maximal rank.\\
Case II) The dimension of $\fr{h}$ is at most one.\\
Case III) $\fr{h}=\fr{su}(2)$.\\

In Case I), $H$ has maximal rank in $G$.  From the main theorem in \cite{AlNi} it follows that the only compact, simply connected spaces $G/H$ with $G$ semisimple of rank two and $\op{rank}(H)=\op{rank}(G)$ that admit non-standard g.o. metrics are the flag manifolds $SO(5)/U(2)$ and $Sp(2)/(U(1)\times Sp(1))$, which are diffeomorphic (\cite{AlAr}).  Moreover, the $Sp(2)$-invariant metrics on $Sp(2)/(U(1)\times Sp(1))$ are induced from inner products of the form

\begin{equation*}g^{\lambda}=\lambda\left.(-Q)\right|_{\mathcal{M}_F\times \mathcal{M}_F} +\left.(-Q)\right|_{\mathcal{M}_B\times \mathcal{M}_B},\end{equation*}

\noindent where $Q$ is the Killing form of $\fr{sp}(2)$, $\mathcal{M}_F=T_o\big((Sp(1)\times Sp(1)/(Sp(1)\times U(1))\big)=T_o(Sp(1)/U(1))=T_o(\mathbb C P^1)$ and $\mathcal{M}_B=T_o\big(Sp(2)/(Sp(1)\times Sp(1))\big)=T_o(\mathbb H P^1)$ (see for example \cite{AlAr}, \cite{BeNis} or \cite{Ta}).  Hence the metrics $g^{\lambda}$ are deformations of the standard metric $g^1$ along the fiber $\mathbb C P^1$ of the Hopf fibration of $\mathbb C P^3=Sp(2)/(Sp(1)\times U(1))$ on $\mathbb H P^1$.  Any metric $g^{\lambda}$ is a g.o. metric.  The space $\big(Sp(2)/(U(1)\times Sp(1)),g^{\lambda}\big)$ yields part (iii) of Theorem \ref{main}.  

In Case II), $H$ is abelian and the main theorem in \cite{So2} implies that any g.o. metric on $G/H$ is normal.  Therefore, it remains to investigate Case III). We will do so for each of the distinct cases $\fr{g}=\fr{su}(3)$, $\fr{so}(4)$, $\fr{sp}(2)$, $\fr{g}_2$, for which we refer to Table~II. 

\subsection{The case $\fr{h}=\fr{su}(2)$ and $\fr{g}=\fr{su}(3)$.}  The corresponding simply connected group is $G=SU(3)$.  According to Table~II, there are two distinct embeddings of $\fr{h}$ in $\fr{g}$ up to conjugation by inner automorphisms. 

 Firstly, assume that $\fr{h}=\op{span}_{\mathbb R}\{\iu(H_{\alpha}+H_{\beta})\}\oplus \fr{m}_{\alpha+\beta}$.  The corresponding simply connected space $G/H$ is the weakly symmetric sphere $S^5=SU(3)/SU(2)$.  Taking into account the explicit expression of $\fr{m}$ in Table~II, the definition of the spaces $\fr{m}_{\gamma}$, relations \eqref{ma} as well as the Lie bracket relations in Lemma \ref{BracketLem} along with the values of the integers $N_{\gamma,\delta}$ in subsection \ref{rootsection}, we deduce that the isotropy algebra representation $\op{ad}^{\fr{g}/\fr{h}}:\fr{h}\rightarrow \op{End}(\fr{m})$ (see subsection \ref{IsotSub}) admits the isotypic decomposition $\fr{m}=\fr{p}\oplus \fr{q}$, where $\fr{p}=\op{span}_{\mathbb R}\{\iu (H_{\alpha}-H_{\beta})\}$ and $\fr{q}=\fr{m}_{\alpha}\oplus \fr{m}_{\beta}$.  The submodules $\fr{p}$ and $\fr{q}$ are $\op{ad}_{\fr{h}}$-irreducible and hence any $G$-invariant metric on $G/H$ is induced (up to homothety) from an inner product $g^{\lambda}$ of the form

 \begin{equation*}g^{\lambda}=\lambda\left.(-Q)\right|_{\fr{p}\times \fr{p}}+\left.(-Q)\right|_{\fr{q}\times \fr{q}}.\end{equation*}

 \noindent We remark that the metrics $g^{\lambda}$ exhaust the both the $SU(3)$-invariant and the $U(3)$-invariant metrics on $S^5=SU(3)/SU(2)=U(3)/U(2)$ (\cite{Nik0}).  The latter are deformation metrics corresponding to the fibration $U(3)/U(2)\rightarrow U(3)/(U(2)\times U(1))$ with fiber $U(1)$. Here $\fr{p}=T_o(U(1))=T_oS^1$, $\fr{q}=T_o\big(U(3)/(U(2)\times U(1))\big)=T_o\mathbb C P^2$ and thus the metrics $g^{\lambda}$ can be considered as a one-parameter family of deformations of the standard metric $g^1$ along the fiber $S^1$ of the Hopf fibration $S^5\rightarrow \mathbb C P^2$.  According to \cite{KoVa} (Theorem 4.4) or \cite{Nik0}, any of the above metrics is a g.o. metric. 
 
 Secondly, assume that $\fr{h}=\op{span}_{\mathbb R}\{F_{\alpha}+F_{\beta},G_{\alpha}+G_{\beta},\iu (H_{\alpha}+H_{\beta})\}$.  Taking into account the explicit expression of $\fr{m}$ in Table~II, and using the same reasoning as above, we deduce that the isotropy representation $\op{ad}^{\fr{g}/\fr{h}}:\fr{h}\rightarrow \op{End}(\fr{m})$ is irreducible.  In fact, $G/H$ is the symmetric space $SU(3)/SO(3)$.  Hence any $G$-invariant metric on $G/H$ is the standard metric. 
 
 We conclude that the only non-normal simply connected g.o. space of the form $(SU(3)/H, g)$, with $\fr{h}=\fr{su}(2)$, is the sphere $(SU(3)/SU(2),g^{\lambda})$.  This yields part (ii) of Theorem \ref{main}.

 \subsection{The case $\fr{h}=\fr{su}(2)$ and $\fr{g}=\fr{so}(4)=\fr{su}(2)\oplus \fr{su}(2)$.}  According to Table~II, there are three distinct embeddings of $\fr{h}$ in $\fr{g}$ up to conjugation by inner automorphisms.  If $\fr{h}=\op{span}_{\mathbb R}\{\iu H_{\alpha}\}\oplus \fr{m}_{\alpha}$ or $\fr{h}=\op{span}_{\mathbb R}\{\iu H_{\beta}\}\oplus \fr{m}_{\beta}$ then $\fr{m}=\fr{su}(2)$ and the corresponding simply connected space $G/H$ is diffeomorphic to $SU(2)$.  In that case, any g.o. metric is bi-invariant and hence normal.  If $\fr{h}=\op{span}_{\mathbb R}\{F_{\alpha}+F_{\beta},G_{\alpha}+G_{\beta},\iu (H_{\alpha}+H_{\beta})\}$ then $\fr{m}$ is $\op{ad}_{\fr{h}}$-irreducible, the corresponding simply connected space $G/H$ is diffeomorphic to the sphere $SO(4)/SO(3)$ and any $G$-invariant metric on $G/H$ is normal.  We conclude that for any g.o. space $(G/H, g)$, the metric $g$ is normal. 

\subsection{The case $\fr{h}=\fr{su}(2)$ and $\fr{g}=\fr{sp}(2)$.}  The corresponding simply connected group is $G=Sp(2)$.  According to Table~II, there are three distinct embeddings of $\fr{h}$ in $\fr{g}$ up to conjugation by inner automorphisms.

 Firstly, assume that $\fr{h}=\op{span}_{\mathbb R}\{\iu(H_{\alpha}+H_{\beta})\}\oplus \fr{m}_{\alpha+2\beta}$.  The corresponding simply connected  space $G/H$ is the sphere $S^7=Sp(2)/Sp(1)$.  Taking into account the explicit expression of $\fr{m}$ in Table~II, relations \eqref{ma} as well as the Lie bracket relations in Lemma \ref{BracketLem} along with the values of the integers $N_{\gamma,\delta}$ in subsection \ref{rootsection},
we deduce that the isotropy representation $\op{ad}^{\fr{g}/\fr{h}}:\fr{h}\rightarrow \op{End}(\fr{m})$ admits the isotypic decomposition $\fr{m}=\fr{p}\oplus \fr{q}$ where $\fr{p}=\fr{n}_{\fr{g}}(\fr{h})/\fr{h}=\op{span}_{\mathbb R}\{\iu H_{\alpha}\}\oplus \fr{m}_{\alpha}=\fr{su}(2)$ and $\fr{q}=\fr{m}_{\beta}\oplus\fr{m}_{\alpha+\beta}$.  
 
  Assume that $\Lambda\in \op{End}(\fr{m})$ is a metric endomorphism corresponding to a g.o. metric on $G/H$.  The isotypic component $\fr{q}$ is $\op{ad}_{\fr{h}}$-irreducible and thus $\left.\Lambda\right|_{\fr{q}}=\mu\op{Id}$.  Moreover, Lemma \ref{DualNormalizer} implies that $\left.\Lambda\right|_{\fr{p}}$ defines a bi-invariant metric on $N_G(H)/H$.  Given that $\fr{p}$ is a simple Lie algebra, Lemma \ref{GOLieGroups} yields $\left.\Lambda\right|_{\fr{p}}=\lambda\op{Id}$.  We conclude that a g.o. metric on $G/H$ is necessarily induced (up to homothety) from an inner product of the form 
 
\begin{equation*}g^{\lambda}=\lambda\left.(-Q)\right|_{\fr{p}\times \fr{p}}+\left.(-Q)\right|_{\fr{q}\times \fr{q}}.\end{equation*}

\noindent We have $\fr{q}=T_o\big(Sp(2)/(Sp(1)\times Sp(1))\big)=T_o(\mathbb H P^1)$, $\fr{p}=T_o\big(Sp(1)\times Sp(1)/Sp(1)\big)=T_o(Sp(1))=T_o(S^3)$ and hence the metrics $g^{\lambda}$ are deformations of the standard metric $g^1$ along the fibers $S^3$ of the Hopf fibration of $S^7$ on the symmetric space $\mathbb H P^1$.  According to \cite{Nik0}, Section 3, the metrics $g^{\lambda}$ exhaust the $Sp(2)\times Sp(1)$-invariant metrics on $Sp(2)/Sp(1)$ and they are geodesic orbit (see Theorem 1 in \cite{Nik0} or \cite{Ta}). 

Now assume that $\fr{h}=\op{span}_{\mathbb R}\{\iu(2H_{\alpha}+H_{\beta})\}\oplus \fr{m}_{\alpha+\beta}$.  Taking into account the explicit expression of $\fr{m}$ in Table~II, relations \eqref{ma} as well as the Lie bracket relations in Lemma \ref{BracketLem} along with the values of the integers $N_{\gamma,\delta}$ in subsection \ref{rootsection}, we deduce that the isotropy representation $\op{ad}^{\fr{g}/\fr{h}}:\fr{h}\rightarrow \op{End}(\fr{m})$ admits the isotypic decomposition $\fr{m}=\fr{p}\oplus \fr{q}$ where $\fr{p}=\op{span}_{\mathbb R}\{\iu H_{\beta}\}$ and $\fr{q}=\fr{m}_{\alpha}\oplus \fr{m}_{\beta}\oplus \fr{m}_{\alpha+2\beta}$.  Moreover, both components $\fr{p},\fr{q}$ are $\op{ad}_{\fr{h}}$-irreducible.  Taking into account the list of the non-normal simply connected geodesic orbit spaces with two irreducible isotropy summands in \cite{CheNi}, and in particular Theorem 2, we deduce that there is no such space with $G=Sp(2)$ and $\fr{h}=\fr{su}(2)$.  Hence any g.o. metric on the corresponding space $G/H$ is normal (and in fact standard).

Finally, assume that $\fr{h}=\op{span}_{\mathbb R}\{2F_{\alpha}+\sqrt{3}F_{\beta},2G_{\alpha}+\sqrt{3}G_{\beta},\iu (4H_{\alpha}+3H_{\beta})\}$.  Again taking into account the explicit expression of $\fr{m}$ in Table~II, relations \eqref{ma} as well as the Lie bracket relations in Lemma \ref{BracketLem} along with the values of the integers $N_{\gamma,\delta}$ in subsection \ref{rootsection}, we deduce that $\fr{m}$ is $\op{ad}_{\fr{h}}$-irreducible and hence the space $G/H$ is isotropy irreducible (in fact it is the isotropy irreducible space $Sp(2)/SU(2)$ \cite{Wo}).  As a result, any $G$-invariant metric on $G/H$ is standard.

We conclude that the only non-normal simply connected g.o. space of the form $(Sp(2)/H, g)$, with $\fr{h}=\fr{su}(2)$, is the sphere $(Sp(2)/Sp(1),g^{\lambda})$.  This yields part (iiv) of Theorem \ref{main}.  It remains to show that for any simply connected space of the form $(G_2/H,g)$ with $\fr{h}=\fr{su}(2)$, the metric $g$ is normal (i.e. standard).

\subsection{The case $\fr{h}=\fr{su}(2)$ and $\fr{g}=\fr{g}_2$.}  The corresponding simply connected group is $G=G_2$.  According to Table~II, there are four distinct embeddings of $\fr{h}$ in $\fr{g}$ up to conjugation by inner automorphisms.  Before we examine them case by case, we will need the following preliminary result.  

\begin{prop}\label{mainargument}
Let $H$ be a connected subgroup of $G_2$ with Lie algebra $\fr{h}=\fr{su}(2)$ and consider the $Q$-orthogonal reductive decomposition $\fr{g}_2=\fr{h}\oplus \fr{m}$.  Assume that the isotypic component $\fr{p}=\{X\in \fr{m}: [a,X]=0\ \ \makebox{for all} \ \ a\in \fr{h}\}$ of the isotropy representation $\op{ad}^{\fr{g}_2/\fr{h}}:\fr{h}\rightarrow \op{End}(\fr{m})$ is isomorphic to $\fr{su}(2)$.  Then the only Riemannian $G_2$-invariant g.o. metric on $G_2/H$ is the standard metric.
\end{prop}

To prove Proposition \ref{mainargument}, we will need the following lemma.

\begin{lemma}\label{Uniq}
Up to conjugation by inner automorphisms, there is a unique subalgebra of $\fr{g}_2$ isomorphic to $\fr{su}(2)\oplus \fr{su}(2)=\fr{so}(4)$.
\end{lemma}

\begin{proof} Up to conjugation by inner automorphisms, the maximal Lie subalgebras of $\fr{g}_2$ that have maximal rank are the algebras
$\fr{su}(2)\oplus \fr{su}(2)$ and $\fr{su}(3)$ (\cite{BoSi}).  If there existed another embedding of $\fr{su}(2)\oplus \fr{su}(2)$ in $\fr{g}_2$ then the algebra would be contained in a maximal subalgebra of maximal rank, i.e. in $\fr{su}(3)$.  But this is a contradiction since the only regular subalgebras of $\fr{su}(n+1)^{\mathbb C}=A_n$ are of the form $A_{k_1}+\cdots +A_{k_s}$ where $\sum_{i=1}^s(k_i+1)=n+1$ (see for example \cite{Dy}).\end{proof}

\noindent \emph{Proof of Proposition \ref{mainargument}.}  Let $\Lambda\in \op{End}(\fr{m})$ be a metric endomorphism corresponding to a g.o. metric and let $\langle \ ,\ \rangle $ be the corresponding inner product on $\fr{m}$.  By Lemma \ref{Katak}, $\fr{p}$ coincides with the Lie algebra $\fr{n}_{\fr{g}_2}(\fr{h})/\fr{h}$ of $N_{G_2}(H)/H$.  Lemma \ref{DualNormalizer} implies that $\left.\Lambda\right|_{\fr{p}}$ defines a bi-invariant metric (and hence a g.o. metric) on $N_{G_2}(H)/H$.  Since $\fr{p}$ is simple, Lemma \ref{GOLieGroups} yields

\begin{equation}\label{eq1}\left.\Lambda\right|_{\fr{p}}=\mu\op{Id}.\end{equation}

\noindent Consider the $Q$-orthogonal decomposition $\fr{m}=\fr{p}\oplus \fr{q}$.  The space $\fr{q}$ coincides with the tangent space $T_o(G_2/N_{G_2}(H))$ and Corollary \ref{NormalizerCorol} implies that $\left.\Lambda\right|_{\fr{q}}$ defines a g.o. metric on $G_2/N_{G_2}(H)$. On the other hand, the fact that $\fr{n}_{\fr{g}_2}(\fr{h})=\fr{h}\oplus \fr{p}=\fr{su}(2)\oplus \fr{su}(2)=\fr{so}(4)$ along with the fact that the embedding of $\fr{su}(2)\oplus \fr{su}(2)$ in $\fr{g}_2$ is unique up to conjugation by an inner automorphism of $\fr{g}_2$ (Lemma \ref{Uniq}) imply that $\fr{q}$ coincides with the tangent space of the symmetric space $G_2/SO(4)$.  Thus $\fr{q}$ is $\op{ad}_{\fr{n}_{\fr{g}_2}(\fr{h})}$-irreducible and hence the g.o. metric $\left.\Lambda\right|_{\fr{q}}$ satisfies $\left.\Lambda\right|_{\fr{q}}=\lambda\op{Id}$.  Along with Equation \eqref{eq1}, we conclude that 

\begin{equation*}\langle \ ,\ \rangle=\lambda\left.(-Q)\right|_{\fr{q}\times \fr{q}}+\mu\left.(-Q)\right|_{\fr{p}\times \fr{p}}.\end{equation*}
   
\noindent Here $\fr{q}$ coincides with the tangent space at the origin of the symmetric space $G_2/SO(4)$ and $\fr{p}$ coincides with the tangent space at the origin of the fiber $SO(4)/(H\cap SO(4))^0$, where we note that the Lie algebra of the compact, connected group $(H\cap SO(4))^0$ is also $\fr{h}$.  By part (ii) of Proposition \ref{UnivCover}, the metric $\langle \ ,\ \rangle$ on $G_2/H$ also defines a $G_2$-invariant g.o. metric on the space $G_2/(H\cap SO(4))^0$, which is fibered over the irreducible symmetric space $G_2/SO(4)$.  Besides, the metric $\langle \ ,\ \rangle$  has the form \eqref{FibForm}.  Taking into account the classification table of the g.o. spaces fibered over irreducible symmetric spaces in \cite{Ta}, we deduce that $\langle \ ,\ \rangle$ is necessarily the standard metric on $G_2/(H\cap SO(4))^0$, or equivalently, $\Lambda=\lambda\op{Id}$.  Therefore, $\Lambda$ also defines the standard metric on $G_2/H$.\qed\\

We now proceed to examine each of the four cases $\fr{h}=\fr{su}(2)$ and $\fr{g}=\fr{g}_2$ in Table~II.  Firstly, assume that $\fr{h}=\op{span}_{\mathbb R}\{\iu H_{\alpha}\}\oplus \fr{m}_{\alpha}$.  Taking into account the explicit expression of $\fr{m}$ in Table~II, relations \eqref{ma} as well as the Lie bracket relations in Lemma \ref{BracketLem} along with the values of the integers $N_{\gamma,\delta}$ in subsection \ref{rootsection}, we deduce that the isotypic component 

\begin{equation*}\fr{n}_{\fr{g}}(\fr{h})/\fr{h}=\{X\in \fr{m}:[a,X]=0\ \ \makebox{for all} \ \ a\in \fr{h}\},\end{equation*}

\noindent of $\op{ad}^{\fr{g}/\fr{h}}$ coincides with the Lie algebra $\op{span}_{\mathbb R}\{\iu H_{3\alpha+2\beta}\}\oplus \fr{m}_{3\alpha+2\beta}$ which is isomorphic to $\fr{su}(2)$.  Proposition \ref{mainargument} then implies that any g.o. metric on the corresponding space $G_2/H$ is standard.

Secondly, assume that $\fr{h}=\op{span}_{\mathbb R}\{\iu H_{\beta}\}\oplus \fr{m}_{\beta}$.  Using the same arguments as above we deduce that the isotypic component $\fr{p}=\fr{n}_{\fr{g}}(\fr{h})/\fr{h}$ coincides with the Lie algebra $\op{span}_{\mathbb R}\{\iu H_{2\alpha+\beta}\}\oplus \fr{m}_{2\alpha+\beta}$, which is isomorphic to $\fr{su}(2)$.  Again, Proposition \ref{mainargument} implies that any g.o. metric on the corresponding space $G_2/H$ is standard.

Now assume that $\fr{h}=\op{span}_{\mathbb R}\{\sqrt{2}(F_{3\alpha+2\beta}-F_{\beta}),\sqrt{2}(G_{3\alpha+2\beta}+G_{\beta}),2\iu H_{3\alpha+\beta}\}$.  Taking into account the explicit expression of $\fr{m}$ in Table~II, relations \eqref{ma} as well as the Lie bracket relations in Lemma \ref{BracketLem} along with the values of the integers $N_{\gamma,\delta}$ in subsection \ref{rootsection}, we deduce that the isotropy representation $\op{ad}^{\fr{g}/\fr{h}}:\fr{h}\rightarrow \op{End}(\fr{m})$ admits the isotypic decomposition $\fr{m}=\fr{p}\oplus \fr{q}$, where

\begin{equation*}\fr{p}=\op{span}_{\mathbb R}\{\sqrt{2}(F_{3\alpha+2\beta}+F_{\beta}),\sqrt{2}(G_{3\alpha+2\beta}-G_{\beta}), 2\iu H_{\alpha+\beta}\}\oplus \fr{m}_{3\alpha+\beta}\ \  \makebox{and} \ \ \fr{q}=\fr{m}_{\alpha}\oplus \fr{m}_{\alpha+\beta}\oplus \fr{m}_{2\alpha+\beta}.\end{equation*}

\noindent Moreover, the submodules $\fr{p},\fr{q}$ are $\op{ad}_{\fr{h}}$-irreducible.  Taking into account the classification of the non-normal simply connected geodesic orbit spaces $(G/H,g)$ with two irreducible isotropy summands in \cite{CheNi}, and in particular Theorem 2, we deduce that there is no such space with $G=G_2$.  Hence any g.o. metric on the corresponding space $G_2/H$ is normal (in fact standard).

  Finally assume that $\fr{h}=\op{span}_{\mathbb R}\{\sqrt{6}F_{\alpha}+\sqrt{10}F_{\beta},\sqrt{6}G_{\alpha}+\sqrt{10}G_{\beta},14\iu H_{9\alpha+5\beta}\}$.  Again taking into account the explicit expression of $\fr{m}$ in Table~II, relations \eqref{ma} as well as the Lie bracket relations in Lemma \ref{BracketLem} along with the values of the integers $N_{\gamma,\delta}$ in subsection \ref{rootsection}, we deduce that $\fr{m}$ is $\op{ad}_{\fr{h}}$-irreducible and hence the corresponding space $G_2/H$ is isotropy irreducible (see also \cite{Wo}).  As a result, any $G$-invariant metric on $G_2/H$ is standard.  This concludes the proof of Theorem \ref{main}.

\end{document}